\documentclass[a4paper,11pt,english,leqno]{amsart}
\usepackage[utf8x]{inputenc}
\usepackage{pgf,tikz,amssymb}
\usepackage[colorlinks=true]{hyperref}
\usepackage[alphabetic,backrefs]{amsrefs}
\usetikzlibrary{arrows}

\theoremstyle{plain}
\newtheorem{thm}{Theorem}[section]
\newtheorem{lem}[thm]{Lemma}
\newtheorem{prop}[thm]{Proposition}
\newtheorem{cor}[thm]{Corollary}

\theoremstyle{definition}
\newtheorem{defn}[thm]{Definition}

\theoremstyle{remark}
\newtheorem*{remark}{Remark}

\newcommand{\R}{\mathbb{R}}
\newcommand{\M}{\mathcal{M}}
\renewcommand{\P}{\mathcal{P}}

\begin{document}

\title[Designs on compact algebraic manifolds]{Asymptotically optimal designs 
on compact algebraic manifolds}
\author{Uju\'e Etayo, Jordi Marzo and Joaquim Ortega-Cerd\`{a}}

\address{Departamento de Matem\'aticas, Estad\'{\i}stica y Computaci\'on, 
Universidad de Cantabria,  Avd. Los Castros 44, 39005 Santander, Spain}
\email{mariadeujue.etayo@unican.es}

\address{Departament de Matem\`atiques i Inform\`atica, Universitat de 
Barcelona \& Bar\-ce\-lo\-na Graduate School of Mathematics, Gran Via 585, 
08007, Barcelona, Spain}
\email{jmarzo@ub.edu}

\address{Departament de Matem\`atiques i Inform\`atica, Universitat de 
Barcelona \& Bar\-ce\-lo\-na Graduate School of Mathematics, Gran Via 585, 
08007, Barcelona, Spain}
\email{jortega@ub.edu}

\date{\today{}}
\thanks{This research has been partially supported by MTM2014-57590-P and 
MTM2014-51834-P grants by the Ministerio de Econom\'{\i}a y Competitividad, 
Gobierno de Espa\~na and by the Ge\-ne\-ra\-li\-tat de Catalunya (project 2014 
SGR 289).}

\begin{abstract}
We find $t$-designs on compact algebraic manifolds with a number of points 
comparable to the dimension of the space of polynomials of degree $t$ on the 
manifold. This generalizes results on the sphere by Bondarenko, Radchenko and 
Viazovska in \cite{BRV13}. Of special interest is the particular case of the 
Grassmannians where our results improve the bounds that had been proved 
previously.
\end{abstract}

\maketitle


\section{Introduction}

Given a real affine algebraic manifold $\M$ endowed with the normalized Lebesgue 
measure $\mu_{\M}$, we say that a collection of $N$ points $x_1,\ldots, x_N\in 
\M$ is a $t$-design, also called averaging set or Chebyshev quadrature 
formula,  if 
\[
\int_\M P(x)\, d\mu_{\M}(x) = \frac 1{N} \sum_{i= 1}^N P(x_i),
\]
for all polynomials $P(x)$ of total degree less or equal than $t$. 

The results for Chebyshev quadrature on the interval are classic, see for 
example the survey paper by Gautschi and the commentaries by 
Korevaar, \cite{Gau14}.

In the case of the sphere, these sets are known as spherical designs and were 
introduced by Delsarte, Goethals and Seidel \cite{DGS77}  in the context of 
algebraic combinatorics on spheres. 

Since then, spherical designs have gained popularity in different areas of 
mathematics, ranging over geometry, algebraic and geometric combinatorics, and 
numerical analysis. See for instance the review \cite{BB09} by Bannai and 
Bannai or the more recent review of Brauchart and Grabner 
\cite{Brauchart2015293}.

Very early the notion of $t$-designs was considered in other contexts beyond the 
sphere. The projective designs were introduced by Hoggard and further 
investigated by Lyubich and Shalatova, \cite{LS04}, as a tool 
to embed isometrically $\ell^2(\R^n)$ in $\ell^{2t}(\R^m)$.

When they are separated, the points in $t$-designs tend to be evenly 
distributed along the sphere or the projective space, and they are close to optimal 
for other problems, like minimization of Newtonian energy as observed in 
\cite{HL08} and \cite{BRV15}. 

For this reason, in order to have evenly distributed subspaces, the 
$t$-designs on Grassmannians were considered by  Bachoc, Coulageon and Nebe in 
\cite{BCN02}. In \cite{BEG16} Breger, Ehler and Gr\"af make some numerical study of approximation problems in 
Grassmannians and make the observation that 
two possible notions of 
designs in Grassmannians (one of them using polynomials as we did and another 
when one replaces polynomials by eigenvectors of the Laplacian as in 
\cite{BCN02}) are indeed related. 

The existence of designs in a very general setting was proved by Seymour and 
Zaslavsky in \cite{SZ84}, see also \cite{A88}. One challenging problem is to 
find designs with as few points as possible for every $t$.

In \cite{Kane15} Kane studied the existence of designs in path-connected spaces. In particular
for the Grasmmanian $G(k, \mathbb{R}^n)$, 
linear subspaces of $\mathbb{R}^n$ of dimension $k,$
he proved that one can find $t$-designs with a number of points of order $O(t^{2k(n−k)})$ 
and he conjectures that it should be of order $O(t^{k(n−k)})$ since this is 
the dimension of the corresponding space of polynomials.

A breakthrough on the existence of designs of size of optimal asymptotic order
was obtained in \cite{BRV13}, where the conjecture by 
Korevaar and Meyers was settled. There are $t$-designs with $O(t^d)$ points in 
the $d$-dimensional sphere. This is asymptotically the best rate possible. The 
proof is obtained by a fixed point theorem, thus it is not constructive. On the 
other hand their method is very flexible, so they could improve it in 
\cite{BRV15} to obtain separated spherical designs with the same cardinality.

As a further evidence of the flexibility of their method we are going to adapt 
it to the setting of a general algebraic manifold and obtain sharp estimates on 
the number of points needed for a $t$-design. The two main ingredients that are 
needed to adapt their method to this setting are the construction of area 
regular partitions in manifolds (this has been recently proved by Gigante and 
Leopardi in \cite{gigante2015diameter}) and the existence of a sampling-type 
inequalities (Marcinkiewicz-Zygmund inequalities) for polynomials for sequences of 
points sufficiently close and separated. This last ingredient essentially follows from a 
Bernstein-type inequality for polynomials in algebraic varieties proved in 
\cite{berman2015sampling}. If we specialize our results to the Grassmannians we 
confirm the conjecture in \cite{Kane15}.


\section{Definitions and Main results}

Let $\M$ be a smooth, connected and compact affine algebraic manifold of dimension 
$d$ in $\R^n$ 
\[
\M=\left\{  x\in \R^n \;: \; p_1(x)=\cdots =p_r(x)=0 \right\},
\] 
where $p_{1},\ldots, p_{r}\in \mathbb{R}[X]$ are polynomials with real 
coefficients and the normal space at $x\in \M,$ generated by $\nabla 
p_1(x),\ldots ,\nabla p_r(x),$ is of dimension $n-d$. We consider in $\M$ the 
$d-$dimensional Hausdorff measure (i.e. the Lebesgue measure) $\mu_{\M}$, 
normalized by $\mu_{\M}(\M) = 1$. We denote as $d(x,y)$ the geodesic distance 
between $x,y \in \M$.

Let $X \subset \mathbb{C}^{n}$ be the complexification of $\M$, i.e. the 
complex zero set of the real polynomials $p_{1}(x),\ldots, p_{r}(x)$. The Lebesgue measure in $X$ will be denoted by $\mu_X$.

The space of real algebraic polynomials on $\M$ of total degree at most $t$,  
denoted by $\mathcal{P}_t=\mathcal{P}_t(\M)$ is the restriction to $\M$ of the 
space of real polynomials in $n$ variables.  The dimension of the space 
$\mathcal{P}_t(\M)$ is given by the Hilbert polynomial and it satisfies:
\[
\operatorname{dim}\mathcal{P}_t(\M)=\operatorname{deg}(\M) t^d+O(t^{d-1}).
\]

Let $\P_t^0$ be the Hilbert space of polynomials  in $\mathcal{P}_t$ with zero 
mean 
\begin{equation*}
\int_{\M} P(x)\, d\mu_{\M}(x)=0,
\end{equation*}
with respect to the usual inner product 
\begin{equation*}
\langle  P,Q \rangle
=  \int_{\M} P(x)Q(x) d\mu_{\M}(x),
\end{equation*}
this space has a reproducing kernel i.e. for each $x \in \M$ there exists a  
unique polynomial $K_{x} \in \P_t^0$ 
such that
\begin{equation*}			\label{eq_4}
\langle Q, K_{x} \rangle = Q(x),
\end{equation*}
for all $Q \in \P_t^0$.

It is clear that $x_{1},\ldots,x_{N} \in \M$ is a $t$-design if and only if
\begin{equation}					\label{eq_5}
\sum_{i=1}^{N} K_{x_{i}} 
= 0.
\end{equation}

The existence of designs was already proved in a very general setting in 
\cite{SZ84}, our aim is to show, adapting the techniques from \cite{BRV13}, 
that one can reach the right order, given by the following result.


\begin{prop}
If $x_1,\ldots ,x_N\in \M$ is a $t$-design in $\M$, then $N\gtrsim t^d$. 
\end{prop}

\begin{proof}
Let $s$ be such that $2s=t$, if $t$ is even, and such that $2s+1=t$ otherwise. 
We have that
\[
\int_\M P^2(x) d\mu_{\M}(x)=\frac{1}{N}\sum_{i=1}^N P^2(x_i),\mbox{ for } P\in 
\mathcal{P}_s.
\]
Suppose that $N<\dim \mathcal{P}_s$. Then for $P_1,\ldots ,P_N\in 
\mathcal{P}_s$ 
linearly independent we have that
\[
\det (P_j(x_i))_{i,j=1,\ldots ,N}\neq 0.
\]
Indeed, if the determinant above vanishes, there exist a non-trivial linear 
combination 
\[
 \sum_{j=1}^N \alpha_j P_j(x_i)=0,\qquad i=1,\ldots , N,
\]
and we get a contradiction from
\[
\int_\M \Bigl( \sum_{j=1}^N \alpha_j P_j(x)\Bigr)^2\, d\mu_{\M}(x)=0.
\]
As $\det (P_j(x_i))_{i,j=1,\ldots ,N}\neq 0$, there exist $Q_1,\ldots , Q_N\in 
\mathcal{P}_s$ such that
$Q_j(x_i)=\delta_{ij}$, for $i,j=1,\ldots, N$, and $Q\in  \mbox{span}\{ 
Q_1,\ldots , Q_N \}^{\perp}$. 
Then
\[
0=\int_\M Q(x) Q_j(x) d\mu_{\M}(x)=\frac{1}{N}Q(x_j),\;\;j=1,\ldots ,N
\]
and from $\int_\M Q(x)^2 d\mu_{\M}(x)=0$ we get that $Q(x)=0$.  
\end{proof}

\begin{remark}
If $x_1,\ldots ,x_N\in \M$ is a $t$-design in $\M$, for even $t$, and $N=\dim 
\mathcal{P}_{t/2}(\M)$, it is easy to see that the set of reproducing kernels 
of the space $\mathcal{P}_{t/2}(\M)$ on those points
\[
K_{t/2}(\cdot, x_j),\;j=1,\ldots N,
\] 
form an orthogonal basis of $\mathcal{P}_{t/2}(\M)$ and 
\[
K_{t/2}(x_j,x_j )=\|K_{t/2}(\cdot,x_j ) \|^2_{L^2(\M)}=N,
\] 
for all $j=1,\ldots, N$. The existence (or not) of these, so-called, tight 
designs in a variety $\M$ seems to be a difficult problem.

In the case of the sphere $\mathbb{S}^d$ the (sharp) lower bounds tell us that 
if $x_1,\ldots ,x_N\in \mathbb{S}^d$ is a $t-$design
\[
N\ge \binom{d+s}{d}+\binom{d+s-1}{d}=\dim 
\mathcal{P}_s(\mathbb{S}^d),\;\;\;\;N\ge 2\binom{d+s}{d}
\] 
for $t=2s$ and $t=2s+1$, respectively. For the sphere $\mathbb{S}^d$ there are 
a few tight spherical designs, see \cites{BD79, BD80}, for which these lower 
bound are attained. The tight spherical designs with larger cardinality are the 
kissing tight $4$-design for $\mathbb{S}^{21}$ of $275=\dim 
\mathcal{P}_2(\mathbb{S}^{21})$ points for even $t$, and the $11-$design for 
$\mathbb{S}^{23}$ of $196560$ points from the Leech lattice for odd $t$.

\end{remark}


Our main result is the following theorem where we show the existence of designs 
with cardinality $N$ for all $N\gtrsim \dim \mathcal{P}_t(\M)$.

\begin{thm}						\label{thm_6} 
There is a constant $C_{\M}$ depending only on $\M$ such that for 
each $N \geq C_{\M} t^{d}$ there are  $t$-designs in $\M$ with $N$ 
points.
\end{thm}

Besides the sharp result for the sphere in \cite{BRV13}, Kuijlaars has proved 
on the torus on $\R^3$ the existence of Chebyshev quadratures with $C t^2$ points 
for polynomials of degree $t$, \cite{Kui95}. In \cite{Kane15} the author obtained 
results in the very general setting of path-connected topological spaces. His 
result in our setting provides designs for any $N\gtrsim t^{2d}$ so twice as much as in our 
result.

To prove Theorem~\ref{thm_6}, we follow the strategy of \cite{BRV13}. The main 
ingredients are a result from Brouwer degree theory and Marcinkiewicz-Zygmund 
inequalities for spaces of polynomials $\mathcal{P}_t(\M)$. In \cite{BRV13} the 
authors borrow the Marcinkiewicz-Zygmund inequalities on the sphere from 
\cite{MNW01}*{Theorem 3.1}, see also \cite{DX13}*{Theorem~6.4.4}. We will prove 
the analogue for algebraic polynomials on algebraic varieties.

To state our results we have to define area regular partitions.

A finite family of closed sets $R_{1},\ldots,R_{N}\subset \M$ is an area 
regular partition of $\M$ if 
\[
\mu_{\M}(R_i)=1/N,\;\;\bigcup_{i=1}^N R_i=\M,\;\;\mbox{and}\;\;\mu_{\M}(R_i\cap 
R_j)=0\;\;\mbox{for}\;\;i\neq j.
\]

\noindent The diameter of the partition $\mathcal{R} = \{ R_{1},\ldots,R_{N} \}$
is
\begin{equation*}						\label{eq_43}
\|\mathcal{R}\| = \max_{i=1,\ldots ,N} \max_{x,y\in R_{i}} d(x,y).
\end{equation*}

Following previous constructions for the sphere, it is not difficult to deduce 
the existence of area regular partitions with diameter comparable to $N^{-1/d}$ 
for any compact algebraic variety, see for example \cite{RSZ94} and the 
references therein. The existence of such a partition in our case can be deduced 
also from a recent result by Gigante and Leopardi for Ahlfors regular metric 
measure spaces, see \cite{gigante2015diameter}*{Theorem 2}.

\begin{prop}				\label{thm_2}
For any $N\ge N_0=N_0(\M)\in \mathbb{N}$ there exists an area regular partition 
$\mathcal{R} = \{ R_{1},\ldots,R_{N} \}$ of $\M$, with $\mu_{\M}(R_i)=1/N$ for 
all $1 \leq i \leq N$, and such that	
\begin{equation}					\label{eq_10}
B (c_1 N^{-1/d}) \subset R_i \subset B (c_2 N^{-1/d}),
\end{equation}
where $B(r)$ is a geodesic ball in $\M$ of radius $r>0$ and the constants 
$c_1,c_2$  depend only on $\M$.
\end{prop}


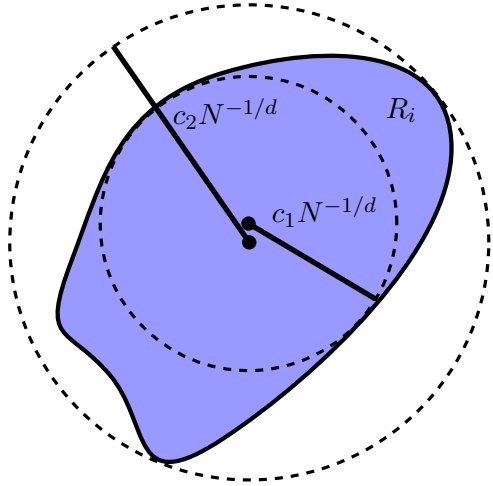
\begin{center}
\begin{tikzpicture}
\filldraw[fill=blue!40!white, draw=black, line width=1.6pt] plot [smooth cycle, 
tension=1] coordinates {(3,3) (6,2) (3,-1.85) (1.5,-1) (1,0.5)};
\draw [line width=1.2pt, dashed] (3.33,1) circle (1.95cm);
\draw [line width=1.2pt,color=black, dashed] (3.34,0.75) circle (3.15cm);
\draw [line width=2.pt,color=black] (3.33,1)-- (5,0);
\draw [line width=2.pt] (3.34,0.75)-- (1.55,3.35);
\draw (5,2.8) node[anchor=north west] {$R_{i}$};
\draw (3.5,1.5)node[anchor=north west] {$c_1 N^{-1/d}$};
\draw (2.2,2.8) node[anchor=north west] {$c_2 N^{-1/d}$};
\begin{scriptsize}
	\draw [fill=black] (3.33,1) circle (2.5pt);
	\draw [fill=black] (3.34,0.75) circle (2.5pt);
\end{scriptsize}
\end{tikzpicture}
\end{center}


Our result about Marcinkiewicz-Zygmund inequalities is the following:

\begin{thm}							\label{thm_3}
There exists a constant $A=A(\M)>0$ such that if $N\ge A t^d$ and $\mathcal{R} 
= \{R_{1},\ldots,R_{N}\}$ is an area regular partition of $\M$ as in 
\eqref{eq_10}.
Then for all $P \in \P_t$
\begin{equation}								
\label{eq_15}
\frac{1}{2} \int_{\M} |P(x)| d\mu_{\M}(x)
\leq
\frac{1}{N} \sum_{i=1}^{N} |P(x_{i})| 
\leq
\frac{3}{2} \int_{\M} |P(x)| d\mu_{\M}(x),
\end{equation}
for any choice of $x_{i} \in R_{i}$, with $1 \leq i \leq N$.
\end{thm}


We will need also Marcinkiewicz-Zygmund inequalities for tangential gradients 
of polynomials. Observe that, unlike for the sphere, for a general variety the 
tangential gradient is not necessarily a polynomial.

\begin{defn}
Given a differentiable function $f$ in $\M\subset \R^n$. The tangential 
gradient of $f$ at a point $x \in \M$, denoted as $\nabla_t f(x)$ is the 
orthogonal projection of the gradient $\nabla f(x)$ onto the tangent space of 
$\M$ at $x$. 
\end{defn}


 \begin{cor}								
\label{cor3.6}
There exists a constant $A=A(\M)>0$ such 
that if $N\ge A t^d$ and $\mathcal{R} = \{R_{1},\ldots,R_{N}\}$ is an area 
regular partition of $\M$ as in \eqref{eq_10}.

Then for all $P \in \mathcal{P}_t$
\begin{equation}					\label{eq_75}
\frac{1}{K_\M} \int_{\M} |\nabla_t P(x)| d\mu_{\M}(x)
\leq
\frac{1}{N} \sum_{i=1}^{N} |\nabla_t P(x_{i})| 
\leq
K_\M \int_{\M} |\nabla_t P(x)| d\mu_{\M}(x),
\end{equation}
where $K_\M=3\sqrt{d}\binom{r}{n-d} C_\M$ for $C_\M>0$ depending only on $\M$ 
and for any choice of $x_{i} \in R_{i}$, with $1 \leq i \leq N$.
\end{cor}


As in \cite{BRV13}, the last ingredient of the proof of Theorem \ref{thm_6} is the following 
result from Brouwer degree theory:

\begin{thm}[\cite{cho2006topological}*{Theorem 1.2.9}] 		\label{thm_1}
Let $f:\R^{n} \longrightarrow \R^{n}$ be a continuous mapping and $\Omega$ an 
open bounded subset, with boundary $\partial \Omega$, such that $0 \in \Omega 
\subset \R^{n}$.
If $\langle x,f(x) \rangle >0 $ for all $x \in \partial\Omega$, 
then there exists $x \in \Omega$ satisfying $f(x) = 0$.
\end{thm}

Defining the convenient mapping from $\mathcal{P}_t$ into itself, this result will give us \eqref{eq_5}.


\section{Proofs}

First we prove the Marcinkiewicz-Zygmund inequalities in the algebraic variety 
$\M$ (Theorem~\ref{thm_3}). Similar results have been obtained also in general 
compact Riemannian manifolds 
for spaces of, so-called, diffusion polynomials (i.e. eigenfunctions of 
elliptic differential 
operators, in particular, for the Laplace-Beltrami operator), 
\cites{FM10,FM11}.  
In the proof we use the 
following result from Berman and Ortega-Cerd\`a 
\cite{berman2015sampling}*{Theorem~10}. This result is 
analogous to Plancherel-Polya inequality for entire functions of 
exponential type, \cite{You01}.


\begin{lem}		\label{lem_3}
There exists a constant $C=C_\M > 0$ such that for all polynomials $P \in \P_t$ 
the 
following inequality holds
\begin{equation*} 							
\label{eq_64}
\int_{U \left( \frac{1}{t} \right)} |P(x)| 
d\mu_{X}(x)
\leq
\frac{C}{t^d} \int_{\M} |P(x)| d\mu_{\M}(x),
\end{equation*}
where $U\left(\frac{1}{t}\right) = \left\lbrace x \in X: d(x, \M)  \leq 
\frac{1}{t}\right\rbrace$.

\end{lem}


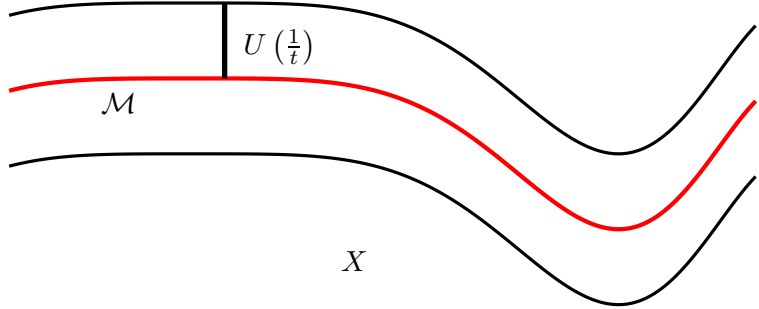
\begin{center}
\begin{tikzpicture}
\draw[line 
width=1.6pt,color=red,smooth,samples=100,
domain=-2.4097289642490645:7.408366288767699] 
plot(\x,{cos(((\x)^(2.0)/10.0)*180/pi)});
\draw[line 
width=1.2pt,smooth,samples=100,domain=-2.4097289642490645:7.408366288767699] 
plot(\x,{cos(((\x)^(2.0)/10.0)*180/pi)-1.0});
\draw[line 
width=1.2pt,color=black,smooth,samples=100,
domain=-2.4097289642490645:7.408366288767699] 
plot(\x,{cos(((\x)^(2.0)/10.0)*180/pi)+1.0});
\draw [line width=2.pt] (0.4245513635086049,0.999837565151902)-- 
(0.4245513635086049,1.999837565151902);
\draw (0.5356996116559645,1.7893456645962826) node[anchor=north west] {$U 
\left( 
\frac{1}{t} \right)$};
\draw (-1.3352958988245887,0.937209095466522) node[anchor=north west] 
{$\mathcal{\M}$};
\draw (1.8231668193628796,-1.1560829113087596) node[anchor=north west] 
{\parbox{1.500167116663118 cm}{$X$}};
\end{tikzpicture}
\end{center}


\begin{proof}[Proof of Theorem~\ref{thm_3}]

It is enough to show that 
\[
\left| \frac{1}{N} \sum_{i=1}^{N} |P(x_{i})|-\int_{\M} |P(x)| d\mu_{\M}(x) 
\right|\le C a\int_{\M} 
|P(x)| d\mu_{\M}(x),
\]
where $C=C(\M)>0$ is a constant depending only $\M$ and $1\ge a A$. Just take 
$A\ge 2C$.

During all the proof, we will call $C$ all the constants depending on $\M$.

Let $N=t^d/a^d$, for some constant $a>0$ such that $1\ge a A$. By assumption
\[
B (a c_1 t^{-1}) \subset R_i \subset B (a c_2 t^{-1}),
\]
and
\[
\frac{2 a c_1}{t} \le \| \mathcal{R} \|\le \frac{2 a c_2}{t}.
\]

Then

\begin{align}				\label{ineq1}
\left| \frac{1}{N} \sum_{i=1}^{N} |P(x_{i})|-\int_{\M} |P(x)| d\mu_{\M}(x) 
\right| 
& \le 
\sum_{i=1}^{N} \left|  \int_{R_i}  |P(x_{i})|- |P(x)|  d\mu_{\M}(x) 
\right|\nonumber
\\
&
\le  
\frac{\| \mathcal{R} \|}{N} \sum_{i=1}^N |\nabla_t P(x_i')| ,
\end{align}
where $x_i'$ is such that $|\nabla_{t}P(x_i')| \geq |\nabla_{t}P(x)|$
for all $x \in R_i$. Observe that we can take $x_i'$ in the ball $B (a c_2 
t^{-1})$ containing $R_i$.

Consider now each $x_{i}'$ as a point of $X$, the complex variety, and apply 
Cauchy's inequality
\begin{equation*}
|\nabla_{t}P(x_{i}')|
\leq
\frac{C}{\left( \frac{1}{t} \right)^{2d+1}} 
\int_{B_X \left( x_{i}', \frac{1}{t}\right)} |P(z)| 
d\mu_{X}(z),  
\end{equation*}
where $C$ is a constant depending on $\M$ and $B_X (x_{i}', t^{-1})$ is a ball 
in $X$. 

We assume that $a<1/2 c_2$ and then $B_X (x_{i}', 2 a c_2 t^{-1})\subset B_X 
(x_{i}', t^{-1})$. With this 
assumption, as $R_i \subset B_X (x_{i}', 2 a c_2 t^{-1})$ and each $R_i$ 
contains a disjoint ball of 
radius $a c_1 t^{-1}$, we get that 
\[
\bigcap_{j=1}^m B_X(x_{i_j}',t^{-1})\neq \emptyset,
\]
implies that
\[
m\le \frac{C}{a^d}.
\]

We can bound the sum on the balls by the integral on a tubular domain around 
$\M$ defined as in Lemma~\ref{lem_3} (see 
\cite{berman2015sampling}) and taking into account the multiplicities:
\[
\sum_{i=1}^N 
\int_{B_X \left( x_{i}', \frac{1}{t}\right)} |P(z)| 
d\mu_{X}(z)\le  \frac{C}{ a^d }
\int_{U\left( \frac{1}{t}\right)} |P(z)| 
d\mu_{X}(z).
\]
Finally, we apply Lemma~\ref{lem_3} and we get that \eqref{ineq1} is bounded by
\[
\frac{\| \mathcal{R} \|}{N} \sum_{i=1}^N |\nabla_t P(x_i')|\le
C \frac{\|\mathcal{R}\|t^{d+1}}{a^d N} \int_{\M} 
|P(x)| d\mu_{\M}(x),
\]
so by using that $N=t^d/a^d$ and the upper bound for $\| \mathcal{R} \|$ we get 
the result.

\end{proof}


To prove the Marcinkiewicz-Zygmund inequalities for the tangential 
gradient (Corollary~\ref{cor3.6}) we use the following inequality for vectors 
of polynomials.

\begin{cor}							\label{cor3.5}
Let $k\in \mathbb N$ be a fixed 
constant. There exists a constant $A=A(\M, k)>0$ such that if $N\ge A t^d$ and 
$\mathcal{R} 
= \{R_{1},\ldots,R_{N}\}$ 
is an area regular partition as in \eqref{eq_10}. 
Then, for all vector of polynomials $Q(x)=(Q_1(x),\ldots , Q_m(x))$ with $m\le 2d$ and
$Q_j(x)\in \mathcal{P}_{t+k}(\M)$ we have that
\begin{equation}
\frac{1}{3\sqrt{d}} \int_{\M} |Q(x)| d\mu_{\M}(x)
\leq
\frac{1}{N} \sum_{i=1}^{N} |Q(x_{i})| 
\leq
3\sqrt{d} \int_{\M} |Q(x)| d\mu_{\M}(x),
\end{equation}
for any election of $x_i\in R_i$.
\end{cor}


\begin{proof}
 Let $A$ be the constant given by the previous theorem when we replace $t+k$ 
instead of 
$t$. Then we use that
\[
|Q(x)|\le \sum_{j=1}^m |Q_j(x)|\le \sqrt{m} |Q (x)|,
\]
and we apply the previous result for each $Q_j(x)$. 
\end{proof}


In \cite{BRV13} this result above was enough because the tangential gradient on 
the sphere of a spherical polynomial can be written as a vector of spherical 
polynomials (i.e. polynomials restricted to the sphere). In our case this is no 
longer the case and we have to be more careful.


\begin{proof}[Proof of Corollary~\ref{cor3.6}]

Let $\M$ be given as the common zero set of the real polynomials $p_1(x),\ldots 
, p_r(x)$.

Since $\M$ is smooth of dimension $d$, for all $x \in \M$ the normal space to 
$\M$ on $x$ is generated by
\[
\nabla p_{i_1}(x), \ldots, \nabla p_{i_{n-d}}(x),
\]
where the index $ i_1<\cdots <i_{n-d}$ (which may depend on $x$) is a subset of 
$\{1,\ldots , r\}$.

Assume that $i_j=j$ for $j=1,\ldots ,n-d$. By the Gram-Schmidt determinant-type 
formula
we obtain an orthogonal basis $u_{1}(x),\ldots,u_{n-d}(x)$ of the normal space 
at $x$ by the 
following determinants
\begin{equation*}
u_{i}(x)
=
\left|
\begin{array}{cccc} 
\left\langle  \nabla p_{1}(x) , \nabla p_{1}(x) \right\rangle & \left\langle  
\nabla p_{2}(x) , \nabla p_{1}(x) \right\rangle & \ldots & \left\langle  \nabla 
p_{i}(x) , \nabla p_{1}(x) \right\rangle 
\\
\left\langle  \nabla p_{1}(x) , \nabla p_{2}(x) \right\rangle & \left\langle  
\nabla p_{2}(x) , \nabla p_{2}(x) \right\rangle & \ldots & \left\langle  \nabla 
p_{i}(x) , \nabla p_{2}(x) \right\rangle 
\\
\vdots
& \vdots
&  \ddots
& \vdots
\\
\left\langle  \nabla p_{1}(x) , \nabla p_{i-1}(x) \right\rangle & \left\langle  
\nabla p_{2}(x) , \nabla p_{i-1}(x) \right\rangle & \ldots & \left\langle  
\nabla p_{i}(x) , \nabla p_{i-1}(x) \right\rangle 
\\
\nabla p_{1}(x)  & \nabla p_{2}(x) & \ldots  & \nabla p_{i}(x)
\\
\end{array}
\right|.
\end{equation*}
Observe that since every $\nabla p_{i}(x)$ is a vector of polynomials, the 
product $\left\langle  \nabla p_{i}(x) , \nabla p_{j}(x) \right\rangle $ is also 
a polynomial an therefore $u_i(x)$ is also a vector of polynomials of total 
degree bounded by a constant depending only on $\M$. The tangential gradient of 
$P$ at $x \in \M$ is then
\begin{equation}
\label{eq_74}
\nabla_{t}P(x) 
= 
\nabla P(x) -
\sum_{i=1}^{n-d} 
\frac{\left\langle \nabla  P(x) , u_{i}(x)\right\rangle u_{i}(x) }{\| u_{i}(x) 
\|^{2}}.
\end{equation}

If there are $n-d$ polynomials defining the normal space to $\M$ in all the 
variety, in particular, for the sphere or any other algebraic hypersurface, the 
result follows because one can apply Corollary~\ref{cor3.5} to the vector of 
polynomials
\[
\left( \prod_{i=1}^{n-d} \| u_i(x) \|^2\right) \nabla_t P(x),
\]
and use that as $\M$ is smooth
\[
0<C_\M^{-1}\le \prod_{i=1}^{n-d} \| u_i(x) \|^2\le C_\M,
\]
for some $C_\M>0$.

Now for any $I\subset \{ 1,\ldots ,r \}$ with $|I|=n-d$ we can define the 
vectors of polynomials $u_j^I(x)$ for $j=1,\ldots, n-d$ (where maybe some of the 
polynomials are zero) and by the previous corollary the 
Marcinkiewicz-Zygmund inequalities hold for
\begin{equation}					\label{eq2}
\left( \prod_{i=1}^{n-d} \| u_{i}^I(x) \|^{2} \right)
\nabla P(x) -
\sum_{i=1}^{n-d} \left\langle \nabla  P(x) , u_{i}^I(x)\right\rangle u_{i}^I(x)
\prod_{j \neq i} \| u_{j}^I(x) \|^{2}.
\end{equation}
Clearly, Marcinkiewicz-Zygmund inequalities hold also taking supremum for the subsets $I\subset \{ 1,\ldots ,r 
\}$ with $|I|=n-d$.

Indeed, now as 
\begin{equation*}
C_\M^{-1} |\nabla_{t} P(x)| 
\le 
\sup_{\substack{I\subset \{1,\ldots , r \}\\|I|=n-d}} \prod_{i = 1}^{n-d} \| 
u_{i}^I(x) \|^{2} |\nabla_{t} P(x)|\le 
C_\M |\nabla_{t} P(x)| ,
\end{equation*}
for some constant $C_\M>0$, the result follows because for $v_I(x)$ as in 
\eqref{eq2}
\[
\binom{r}{n-d}^{-1} \sum_{\substack{I\subset \{1,\ldots , r \}\\|I|=n-d}} 
|v_I(x)|\le 
\sup_{\substack{I\subset \{1,\ldots , r \}\\|I|=n-d}} |v_I(x)|\le 
\sum_{\substack{I\subset \{1,\ldots , r \}\\|I|=n-d}} |v_I(x)|,
\]
for all $x\in \M$.
\end{proof}


Now we define the mapping.

\begin{lem}						\label{prop_4}
There exists a constant $A=A(\M)>0$ such that if $N\ge A t^d$ 
then there exist a continuous mapping
\begin{align*}
\P_t^0 &\rightarrow \M^N\\
P & \mapsto (x_{1}(P),\ldots,x_{N}(P)),
\end{align*}
such that for 
all $P \in \P_t^0$ with $\int_\M |\nabla_t P(x)|d\mu(x)=1$

\begin{equation*}
\label{eq_9}
\Bigl\langle P, \sum_{i=1}^N K_{x_i(P)} \Bigr\rangle= \sum_{i=1}^{N} 
P(x_{i}(P))>0.
\end{equation*}
\end{lem}


\begin{proof}
Let $A=A(\M)>0$ be given by Corollary~\ref{cor3.6}. Let $N\ge A t^d$ and 
$\mathcal{R} = \{R_{1},\ldots,R_{N}\}$ 
be an area regular partition of $\M$ as in \eqref{eq_10}.

Given a polynomial $P$, we define in $\M$ the vector field $X_P = \nabla_t 
P/U_\epsilon (| \nabla_t P|)$, where $U_\epsilon:\R^+\to \mathbb{R}$ is a smooth 
increasing function such that $U_\epsilon(x) = \epsilon/2$ if $0\le x \le 
\varepsilon/2$ and $U_\epsilon(x)=x$ if $x\ge\varepsilon$ for some $\epsilon$ 
fixed. Since $U_\epsilon(x)$ is smooth, the vector field $X_P$ is smooth on $\M$ 
and depends continuously on $P$.

Now for each $1 \leq i \leq N$ we consider the map 
$y_{i}: 
[0, \infty) \longrightarrow \M$ that satisfies the differential equation
\begin{equation}
\label{eq_14}
\begin{cases} 
\frac{\partial}{\partial s}y_{i}(s) = X_P(y_{i}(s)) \\
y_{i}(0) = x_{i}
\end{cases}
\end{equation}
 where $x_{i} \in R_{i}.$
The differential equation changes for each $P \in \P$, thus we will sometimes 
denote $y_i(s)$ as $y_i(P, s)$ to stress the dependence on $P$.
\medskip

\begin{remark}
Note that the quantity $\sum_{i=1}^{N} 
P(x_{i})$ is small since  $\int_{\M} P(x) d\mu(x) = 
0 $. 
In order to increase this quantity, we move from the point $x_{i}$ in the 
direction that increases $P(x_{i})$, that is, the direction 
given by the vector $\nabla_{t}P(x_{i})$.
\end{remark}

Since the vector field $X_P$ is smooth, each $y_{i}$ is well defined and 
continuous in both $P$ 
and $s$. For a fixed $s_0>0$ to be determined, define the continuous mapping
\begin{equation}
\label{eq_17}
\mathcal{P}_t^0\ni P\mapsto \left( 
x_{1}(P) ,\ldots, x_{N}(P)
\right)
= \left( 
y_{1}\left(P, s_0 \right) ,\ldots,
y_{N}\left(P,s_0 \right) 
\right).
\end{equation}


Now, following \cite{BRV13} we split 
\begin{align}			\label{particion}
\frac{1}{N} \sum_{i=1}^{N} P(x_{i}(P)) & 
= \frac{1}{N} \sum_{i=1}^{N} P(x_{i}) 
+\int_{0}^{s_0} \frac{d}{d s} 
\left[ \frac{1}{N} \sum_{i=1}^{N} 
P\left(y_{i}\left(P,s \right)\right)  \right]ds \nonumber
\\
&
\ge 
\int_{0}^{s_0} \frac{d}{d s} 
\left[ \frac{1}{N} \sum_{i=1}^{N} 
P\left(y_{i}\left(P,s \right)\right)   \right]ds-\left| \frac{1}{N} 
\sum_{i=1}^{N} P(x_{i})  \right|
\end{align}

Observe that if $x_i'$ belongs to a ball $B(C_\M \| \mathcal{R}\|)$ containing 
$R_i,$ where $C_\M>0$ is a constant depending only on $\M$ then, defining 
\[
R_i'=(R_i\setminus \cup_{j=1}^N\{x_j'\}) \cup \{x_i'\},
\]
we get an area regular partition with the same properties, i.e. 
\[
B (c_1' N^{-1/d}) \subset R_i' \subset B (c_2' N^{-1/d}),
\]
for some constants $c_1',c_2'$ depending only on $\M$.


As in \eqref{ineq1} and using that $P\in \mathcal{P}_t^0$ has mean zero, 
we get
\[
\left| \frac{1}{N} \sum_{i=1}^{N} P(x_{i}) \right| \le  
\frac{\| \mathcal{R} \|}{N} \sum_{i=1}^N |\nabla_t P(x_i')|,
\]
where $x_i'$ is a point in the ball $B(c_2' N^{-1/d})$ containing $R_i$ where 
$|\nabla_t P(x)|$ attains its maximum. Applying Corollary~\ref{cor3.6} we get 
that
\[
\left| \frac{1}{N} \sum_{i=1}^{N} P(x_{i}) \right|\le K_\M \|\mathcal{R}\|.
\]

For any fixed $0<s<C_\M \| \mathcal{R}\|$ we apply the remark above about the  
area regular partition and Corollary~\ref{cor3.6} to get
\begin{align*}
\frac{d}{d s} & \left[ \frac{1}{N}\sum_{i=1}^{N} P\left(y_{i}\left(P,s 
\right)\right)  \right]ds
=\frac{1}{N}\sum_{i=1}^{N} \frac{|\nabla_t P(y_{i} 
(P,s))|^2}{U_\epsilon(|\nabla_t P(y_{i}(P,s))|)}  
\\
\ge
&
\frac{1}{N}\sum_{i\;:\;|\nabla_t P(y_{i}(P,s))|\ge \epsilon} |\nabla_t 
P(y_{i}(P,s))|\ge \frac{1}{N} \sum_{i=1}^{N} |\nabla_t P(y_{i}(P,s))|-\epsilon
\\
\ge
&
\frac{1}{K_\M}-\epsilon.
\end{align*}

So, finally taking $s_0=3 K_\M^2 \| \mathcal{R}\|$ and 
$\epsilon=\frac{1}{2K_\M}$ we get from \eqref{particion}
\[
\frac{1}{N} \sum_{i=1}^{N} P(x_{i}(P))\ge \frac{K_\M \| \mathcal{R} 
\|}{2}>0.
\]

\end{proof}


\begin{proof}[Proof of Theorem~\ref{thm_6}]
Fix $t$ and define
\[
\Omega =\left\lbrace P \in \P_t^0 : \int_{\M} | \nabla_{t} P(x)| d\mu_{\M}(x) < 
1 \right\rbrace,
\]
which is clearly an open, bounded subset of $\P_t^0$ such that $0 \in  \Omega 
\subset \P_t^0$.
Take $N\ge A t^d$ for $A$ given by the previous lemma and let $x_i(P)$ be the 
points defined for $P\in \partial \Omega$.

From the continuity of 
$P \mapsto (x_{1}(P),\ldots,x_{N}(P))$ it follows that 
\[
\mathcal{P}_t^0 \ni P\mapsto \sum_{i=1}^N K_{x_i(P)},
\]
is continuous
and from Lemma~\ref{prop_4}, for all $P\in\partial \Omega$, 
\[
\bigl\langle P, \sum_{i=1}^N K_{x_i(P)}\bigr\rangle=\sum_{i=1}^N P(x_i(P))>0.
\]

Applying Theorem~\ref{thm_1}
we get the existence of some $Q\in \Omega$ for which
$\sum_{i=1}^N K_{x_i(Q)}=0.$

\end{proof}


\end{document}